\documentclass[11pt,psamsfonts]{amsart}
\usepackage{amsmath}
\usepackage{amsthm}
\usepackage{amssymb}
\usepackage{amscd}
\usepackage{amsfonts}
\usepackage{amsbsy}
\usepackage[latin1]{inputenc}
\usepackage{epsfig,afterpage}
\usepackage[dvips]{psfrag}
\usepackage{psfrag}
\usepackage[all]{xy}
\usepackage{color}

\newcommand{\R}{\ensuremath{\mathbb{R}}}

\newcommand{\N}{\ensuremath{\mathbb{N}}}

\newcommand{\Z}{\ensuremath{\mathbb{Z}}}

\usepackage{overpic}

\newtheorem {theorem} {Theorem}
\newtheorem {proposition}  {Proposition}
\newtheorem {lemma}  {Lemma}
\newtheorem {definition} [theorem] {Definition}
\newtheorem {remark} {Remark}
\newtheorem {example} {Example}

\begin{document}

\title[Chaos and Non Trivial Minimal Sets in PSVFs]{Chaotic Planar Piecewise Smooth Vector Fields With Non Trivial Minimal Sets}

\author[C.A. Buzzi, T. de Carvalho and R.D. Euzébio]
{Claudio A. Buzzi$^1$, Tiago de Carvalho$^2$ and\\ Rodrigo D. Euzébio$^1$}

\address{$^1$ IBILCE--UNESP, CEP 15054--000,
S. J. Rio Preto, S\~ao Paulo, Brazil}

\address{$^2$ FC--UNESP, CEP 17033--360,
Bauru, São Paulo, Brazil}

\email{buzzi@ibilce.unesp.br}

\email{tcarvalho@fc.unesp.br}

\email{rodrigo.euzebio@sjrp.unesp.br}


\date{}
\dedicatory{} \maketitle


\begin{abstract}
In this paper some aspects on chaotic behavior and minimality in planar piecewise smooth vector fields theory are treated. The occurrence of non-deterministic chaos is observed and the concept of orientable minimality is introduced. It is also investigated some relations between minimality and orientable minimality and observed the existence of new kinds of non-trivial minimal sets in chaotic systems. The approach is geometrical and involve the ordinary techniques of non-smooth systems.
\end{abstract}

\section{Introduction}\label{secao introducao}

Piecewise smooth vector fields (PSVFs, for short) have become
certainly one of the common frontiers between Ma\-the\-matics and
Physics or Engineering.
Many authors have contributed to the study of PSVFs  (see
for instance the pioneering work \cite{Fi} or the didactic works \cite{diBernardo-livro,Marco-enciclopedia}, and references therein about details of
these multi-valued vector fields).  In our approach Filippov's convention is considered. So, the vector field of the model is discontinuous across a \textit{switching manifold} and it is possible for its trajectories to be
confined onto the switching manifold itself. The occurrence of such
behavior, known as \textit{sliding motion}, has been reported in a
wide range of applications. We can find important examples in electrical circuits having switches, in mechanical devices in which components collide into each other, in problems with friction, sliding or squealing, among others (see \cite{diBernardo-livro}).

\smallskip

For planar smooth vector fields there is a very developed theory nowadays, mainly in the planar case. In such environment, questions about chaotic behaviour and minimality, for instance, are complety answered. Indeed, the Jordan curve theorem assures that there is no chaotic behaviour in planar systems and the Poincaré-Bendixson theorem says that for a given flow the minimal sets are just equilibria or limit cycles. Nevertheless, in higher dimension, while minimal sets are described by the Denjoy-Schwartz theorem (under some suitable hypothesis $-$ see \cite{Gutierrez}), chaotic systems are massively studied and a final theory is far away from be reached.

\smallskip

A very interesting and useful subject is to study these kind of objects in the PSVF's scenario. Furthermore, we must observe that chaotic behaviour and non-trivial minimality have been understudied in the PSVF's literature. In three-dimensional systems, some results on chaotic behaviour were obtained by Jeffrey and coauthor in \cite{J} and \cite{CJ}. As far as we know, chaos has not been treated in planar PSVFs apart from the particular example exhibited in \cite{J}. In addition, some questions on non-trivial minimality were provided in \cite{BCE} for planar PSVFs.

\smallskip

The specific topic addressed in this paper concerns with the occurrence of chaos in planar PSVFs and some distinct definitions of minimal sets. Moreover, we study the occurrence of non-trivial minimal sets different from those ones presented in \cite{BCE}. We stress that a trivial minimal set is either an equilibrium point  or a closed periodic orbit. For smooth vector fields this is a very important subject because minimal sets are part essential of limit sets. As long as the authors know, a first study about the minimal set theory for PSVFs and a discussion about the validity of the Poincaré-Bendixson theorem for PSVFs is given only in our previous paper \cite{BCE}. In \cite{BCE} we gave a definition of minimal sets for this scenario and exhibited non-trivial minimal sets for planar PSVFs with sliding motion. Morever, we showed that an analogous theorem to the Poincaré-Bendixson one for PSVFs without sliding motion can be achieved.

\smallskip

Following the approach in \cite{BCE}, here we present some special PSVFs and prove the existence of compact invariant sets with chaotic flow. Actually, these sets will be non-trivial minimal sets having no symmetry. We also propose definitions of minimal sets for positive (and negative) flow of PSVFs (or orientable minimality) and study some relations between them and the definition of minimal set stablished in \cite{BCE}. With these new definitions we analyze the occurrence of new kind of non-trivial minimal sets for PSVFs defined in $\R^2$ and the validity of an analogous to the Poincaré-Ben\-dix\-son Theorem.

\smallskip

The paper is organized as follows: In Section \ref{secao preliminares} some concepts of the standard theory on PSVFs and a brief introduction about Filippov systems are introduced. In Section \ref{results} the results of the paper are presented into two parts: first, in Subsection \ref{chaos}, the occurence of chaos in a particular set is verified. Later, in Subsection \ref{classes_minimais} definitions of minimal sets for positive (and negative) flow of PSVFs are stablished, some correspondences between them are studied and the presence of chaotic behaviour in some minimal sets under these definitions are examined. Moreover, a theorem relating orientable minimality and chaos is presented. Finally, in Section \ref{conclusao} a short conclusion of the work developed in the present paper is presented.

\section{Preliminaries}\label{secao preliminares}

Let $V$ be an arbitrarily small neighborhood of $0\in\R^2$ and consider a codimension one manifold $\Sigma$ of $\R^2$ given by
$\Sigma =f^{-1}(0),$ where $f:V\rightarrow \R$ is a smooth function
having $0\in \R$ as a regular value (i.e. $\nabla f(p)\neq 0$, for
any $p\in f^{-1}({0}))$. We call $\Sigma$ the \textit{switching
manifold} that is the separating boundary of the regions
$\Sigma^+=\{q\in V \, | \, f(q) \geq 0\}$ and $\Sigma^-=\{q \in V \,
| \, f(q)\leq 0\}$. Observe that we can assume, locally around the origin of $\R^2$, that
$f(x,y)=y.$

Designate by $\chi$ the space of C$^r$-vector fields on
$V\subset\R^2$, with $r \geq 1$
large enough for our purposes. Call $\Omega$ the space of vector
fields $Z: V \rightarrow \R ^{2}$ such that
\begin{equation}\label{eq Z}
 Z(x,y)=\left\{\begin{array}{l} X(x,y),\quad $for$ \quad (x,y) \in
\Sigma^+,\\ Y(x,y),\quad $for$ \quad (x,y) \in \Sigma^-,
\end{array}\right.
\end{equation}
where $X=(X_1,X_2) , Y = (Y_1,Y_2) \in \chi$.  The trajectories of
$Z$ are solutions of  ${\dot q}=Z(q)$ and we accept it to be
multi-valued at points of $\Sigma$. The basic results of
differential equations in this context were stated by Filippov in
\cite{Fi}, that we summarize next. Indeed, consider Lie derivatives \[X.f(p)=\left\langle \nabla f(p),
X(p)\right\rangle \quad \mbox{ and } \quad X^i.f(p)=\left\langle
\nabla X^{i-1}. f(p), X(p)\right\rangle, i\geq 2
\]
where $\langle . , . \rangle$ is the usual inner product in $\R^2$.

\smallskip

We  distinguish the following regions on the discontinuity set
$\Sigma$:
\begin{itemize}
\item [(i)]$\Sigma^c\subseteq\Sigma$ is the \textit{sewing region} if
$(X.f)(Y.f)>0$ on $\Sigma^c$ .
\item [(ii)]$\Sigma^e\subseteq\Sigma$ is the \textit{escaping region} if
$(X.f)>0$ and $(Y.f)<0$ on $\Sigma^e$.
\item [(iii)]$\Sigma^s\subseteq\Sigma$ is the \textit{sliding region} if
$(X.f)<0$ and $(Y.f)>0$ on $\Sigma^s$.
\end{itemize}

The \textit{sliding vector field}
associated to $Z\in \Omega$ is the vector field  $Z^s$ tangent to $\Sigma^s$
and defined at $q\in \Sigma^s$ by $Z^s(q)=m-q$ with $m$ being the
point of the segment joining $q+X(q)$ and $q+Y(q)$ such that $m-q$
is tangent to $\Sigma^s$ (see Figure \ref{fig def filipov}). It is
clear that if $q\in \Sigma^s$ then $q\in \Sigma^e$ for $-Z$ and then
we  can define the {\it escaping vector field} on $\Sigma^e$
associated to $Z$ by $Z^e=-(-Z)^s$. In what follows we use the
notation $Z^\Sigma$ for both cases.  In our pictures
we represent the dynamics of $Z^\Sigma$ by double arrows.\\

\begin{figure}[!h]
\begin{center}
\psfrag{A}{$q$} \psfrag{B}{$q + Y(q)$} \psfrag{C}{$q + X(q)$}
\psfrag{D}{} \psfrag{E}{\hspace{1cm}$Z^\Sigma(q)$}
\psfrag{F}{\hspace{.7cm}$\Sigma^s$} \psfrag{G}{} \epsfxsize=5.5cm
\epsfbox{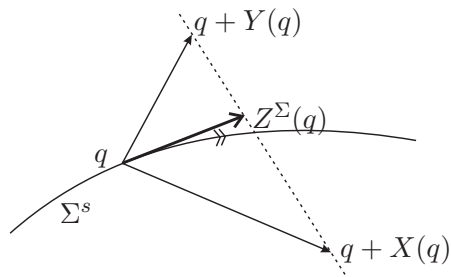} \caption{\small{Filippov's
convention.}} \label{fig def filipov}
\end{center}
\end{figure}

\smallskip

In what follows we present the definition of local and global trajectories for PSVFs. Before that, we remark that a tangency point of system \eqref{eq Z} is characterized by $(X.f(q))(Y.f(q)) =0$. If there exist a characteristic orbit of the vectors fields $X$ or $Y$ reaching $q$ in a finite time, then such tangency is called a \textit{visible tangency}. Otherwise we call $q$ an \textit{invisible tangency}. In addition, a tangency point $p$ is \textit{singular} if $p$ is a invisible tangency for both $X$ and $Y$.  On the other hand, a tangential singularity $p$ is \textit{regular} if it is not singular.

\smallskip

The definition of local trajectory can also be found in \cite{Marcel}.

\begin{definition}\label{definicao trajetorias}
The \textbf{local trajectory (orbit)} $\phi_{Z}(t,p)$ of a PSVF given by \eqref{eq Z} is
defined as follows:
\begin{itemize}
\item For $p \in \Sigma^+ \backslash \Sigma$ and $p \in \Sigma^{-} \backslash \Sigma$ the trajectory is
given by $\phi_{Z}(t,p)=\phi_{X}(t,p)$ and
$\phi_{Z}(t,p)=\phi_{Y}(t,p)$ respectively, where $t \in I$.

\item For $p \in \Sigma^{c}$ such that $X.f(p)>0$, $Y.f(p)>0$ and  taking the
origin of time at $p$, the trajectory is defined as
$\phi_{Z}(t,p)=\phi_{Y}(t,p)$ for $t \in I \cap \{ t \leq 0 \}$ and
$\phi_{Z}(t,p)=\phi_{X}(t,p)$ for $t \in I \cap \{ t \geq 0 \}$. For
the case $X.f(p)<0$ and $Y.f(p)<0$  the definition is the same
reversing time.

\item For $p \in \Sigma^e$ and  taking the
origin of time at $p$, the trajectory is defined as
$\phi_{Z}(t,p)=\phi_{Z^{\Sigma}}(t,p)$ for $t \in I \cap \{ t \leq 0 \}$ and
$\phi_{Z}(t,p)$ is either $\phi_{X}(t,p)$ or $\phi_{Y}(t,p)$ or $\phi_{Z^{\Sigma}}(t,p)$ for $t \in I \cap \{ t \geq 0 \}$. For
the case $p \in \Sigma^s$  the definition is the same
reversing time.

\item For $p$ a regular tangency point and  taking the
origin of time at $p$, the trajectory is defined as
$\phi_{Z}(t,p)=\phi_{1}(t,p)$ for $t \in I \cap \{ t \leq 0 \}$ and
$\phi_{Z}(t,p)=\phi_{2}(t,p)$ for $t \in I \cap \{ t \geq 0 \}$, where each $\phi_{1},\phi_{2}$ is either $\phi_{X}$ or $\phi_{Y}$ or $\phi_{Z^{\Sigma}}$.

\item For $p$ a singular tangency point
    $\phi_{Z}(t,p)=p$ for all $t \in \R$.
\end{itemize}
\end{definition}

\smallskip

The next definitions was stated in \cite{BCE}.

\smallskip

\begin{definition}\label{definicao orbita}
A \textbf{global trajectory (orbit)} $\Gamma_{Z}(t,p_0)$ of $Z\in \chi$ passing through $p_0$ is a union $$\Gamma_{Z}(t,p_0)=\bigcup_{i\in \Z} \{ \sigma_i(t,p_i); t_i \leq t \leq t_{i+1} \}$$of preserving-orientation local trajectories $\sigma_i(t,p_i)$ satisfying $\sigma_i(t_{i+1},p_i)=\sigma_{i+1}(t_{i+1},p_{i+1})=p_{i+1}$ and $t_i \rightarrow \pm \infty$ as $i \rightarrow \pm \infty$. A global trajectory is a \textbf{positive} (respectively, \textbf{negative}) \textbf{global trajectory} if $i \in \N$ (respectively,  $-i \in \N$) and $t_0 = 0$.
\end{definition}

\smallskip

\begin{definition}\label{invariancia}
A set $A\subset\mathbb{R}^{2}$ is $\mathbf{Z}$\textbf{-invariant} if for each $p\in A$ and all global trajectory $\Gamma_{Z}(t,p)$ passing through $p$ it holds $\Gamma_{Z}(t,p) \subset A$.
\end{definition}

\smallskip

\begin{definition}\label{definicao minimal Z}
Consider $\mathbf{Z \in \Omega}$. A set $M\subset\mathbb{R}^{2}$ is $Z$-\textbf{minimal} if
\begin{enumerate}
\item[(i)] $M\neq\emptyset$;
\item[(ii)] $M$ is compact;
\item[(iii)] $M$ is $Z$-invariant;
\item[(iv)] $M$ does not contain proper subset satisfying (i), (ii) and (iii).
\end{enumerate}
\end{definition}

In the next section we present the main results of the paper.

\section{Main Results}\label{results}

\subsection{Non Deterministic Chaos in Planar PSVFs}\label{chaos}

Since the dynamic on sliding and escaping regions are set-valued, following the previous nomenclature of \cite{CJ} and \cite{M}, it is non-deterministic. In fact, as long as we know, the definition of non-deterministic chaos for PSVFs was first introduced in \cite{CJ}, where the authors adapt the classical definition of, for example \cite{M}, to this context. Of course, the definition must contemplate topological transitivity and sensitivity to initial conditions. For this purpose, consider the following definitions:

\begin{definition}\label{definicao transitividade}
 System \eqref{eq Z} is topologically transitive on an invariant set $W$ if for every pair of nonempty, open sets $U$ and $V$ in $W$, there exist $q\in U$, $\Gamma_{Z}^{+}(t,q)$ a positive global trajectory and $t_{0}>0$ such that $\Gamma_{Z}^{+}(t_{0},q) \in V$.
\end{definition}


\begin{definition}\label{definicao sensibilidade}
System \eqref{eq Z} exhibits sensitive dependence on a compact invariant set $W$ if there is a fixed $r>0$ satisfying $r < diam(W)$ such that for each $x\in W$ and $\varepsilon>0$ there exist a $y\in B_{\varepsilon}(x)\cap W$ and positive global trajectories $\Gamma_{x}^{+}$ and $\Gamma_{y}^{+}$ passing through $x$ and $y$, respectively, satisfying
$$
d_{H}(\Gamma_{x}^{+},\Gamma_{y}^{+})=\displaystyle\sup_{a\in\Gamma_{x}^{+},b\in\Gamma_{y}^{+}}d(a,b)>r,
$$
where $diam(W)$ is the diameter of $W$ and $d$ is the Euclidean distance.
\end{definition}


As observed in \cite{CJ}, the two previous definitions coincide for
single-valued flows, making this a natural extension for a set-valued flow. Now we define a non-deterministic chaotic set:

\begin{definition}\label{definicao caotico}
System \eqref{eq Z} is chaotic on a compact invariant set $W$ if it is topologically transitive and exhibits sensitive dependence on $W$.
\end{definition}

%

In what follows we present a chaotic set coming from a non-trivial minimal set.

\begin{theorem}\label{teorema caos}
Consider $Z=(X,Y) \in \Omega$, where $X(x,y)=(1,-2x)$, $Y(x,y)=$ $(-2,4x^{3}-2x)$ and $\Sigma=f^{-1}(0)=\{(x,y)\in\mathbb{R}^{2};y=0\}$. Then the planar PSVF $Z$ is chaotic (see Figure \ref{fig minimal 2}) on the compact invariant set
\begin{equation}\label{equacao conjunto minimal}
\Lambda=\{(x,y)\in\mathbb{R}^{2};-1\leq x\leq 1\;\; \mbox{and}\;\; x^{4}/2-x^{2}/2\leq y\leq 1-x^{2} \}.
\end{equation}
\end{theorem}

\begin{figure}[h]
\begin{center}
\begin{overpic}[height=0.17\paperheight]{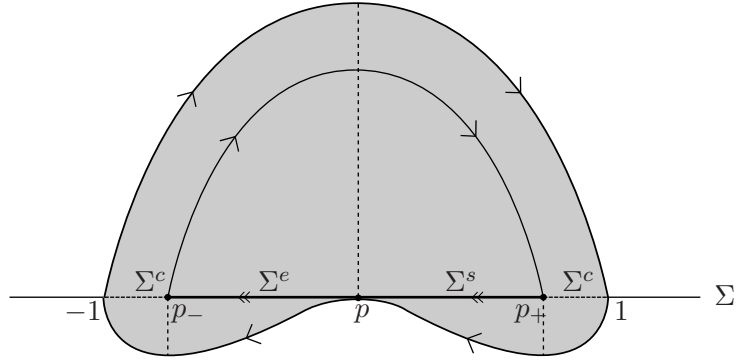}
\put(73.2,6){$p_{+}$} \put(23.5,6){$p_{-}$} \put(50,6){$p$} \put(36,9.6){$\Sigma^{e}$} \put(63,9.6){$\Sigma^{s}$} \put(18,9.6){$\Sigma^{c}$} \put(80,9.6){$\Sigma^{c}$} \put(102,7.6){$\Sigma$} \put(8,5.5){$-1$}  \put(87.5,5.5){$1$}
\end{overpic}
\caption{\small{The $Z$-minimal set $\Lambda$. The PSVF $Z=((1,-2x),(-2,4x^{3}-2x))$ is chaotic on it.}} \label{fig minimal 2}
\end{center}
\end{figure}

Before proving Theorem \ref{teorema caos} we present the following lemma. It will be fundamental in the proof of Theorem \ref{teorema caos}.

\begin{lemma}\label{lemma caos}
Consider the set $\Lambda$ defined in Theorem \ref{teorema caos}. Then, for any $x,y\in\Lambda$, there exist a positive global trajectory $\Gamma^{+}(t,x)$ passing through $x$ and $t_{0}>0$ such that $\Gamma^{+}(t_{0},x)=y$.
\end{lemma}

The previous lemma says that any two points in $\Lambda$ can be connected by some positive global trajectory. Its proof is straighfoward if we observe that a global trajectory of any point in $\Lambda$ meets $p$ for some time $t^{*}$, as the authors argued in \cite{BCE}. Now we prove Theorem \ref{teorema caos}.

\begin{proof}[Proof of Theorem \ref{teorema caos}]
In order to prove that the PSVF Z is topologically transitive on $\Lambda$, we observe that $\Lambda$ is compact and invariant since it is mi\-ni\-mal (see Proposition 1 of \cite{BCE}). Now consider nonempty open sets $U$ and $V$ in $\Lambda$. Since $U$ and $V$ are nonempty, there exist at least an element $\lambda_{1}$ in $U$ and another one $\lambda_{2}$ in $V$. By Lemma \ref{lemma caos}, there exist a positive global trajectory $\Gamma^{+}(t,\lambda_{1})$ passing through $\lambda_{1}$ and $t_{0}>0$ such that $\Gamma^{+}(t_{0},\lambda_{1})=\lambda_{2}\in V$. Consequently the PSVF $Z$ is topologically transitive on the invariant set $\Lambda$.

\smallskip

Now we prove that $Z$ exhibits sensitive dependence on $\Lambda$. Indeed, take $m=diam(\Lambda)$ and consider $r=m/2>0$. Since $r<m$ then there exists two elements $a$ and $b$ in $\Lambda$ such that $d(a,b)>r$. Now consider $x\in\Lambda$, $\varepsilon>0$ and fix $y\in B_{\varepsilon}(x)\cap\Lambda$. By Lemma \ref{lemma caos} there exist positive global trajectories $\Gamma^{+}_{Z}(t,x)$ of $x$ and $\Gamma^{+}_{Z}(t,y)$ of $y$ and numbers $t_{1},t_{2}>0$ such that $\Gamma^{+}_{Z}(t_{1},x)=a$ and $\Gamma^{+}_{Z}(t_{2},y)=b$. Then $d_{H}(\Gamma_{Z}^{+}(t_{1},x),\Gamma_{Z}^{+}(t_{2},y))=d(a,b)>r$ and consequently $Z$ exhibits sensitive dependence on $\Lambda$. Thus the planar PSVF $Z$ is chaotic on the invariant compact set $\Lambda$.
\end{proof}

We should notice that apart of topologically transitive and sensitive dependence, the classical definition of chaos given by Devaney in \cite{Devaney} demands also that periodic trajectories of the considered system are dense in $\Lambda$. However, system $Z$ exhibited in Theorem \ref{teorema caos} also present such property, as we can see in Theorem \ref{orbitas periodicas densas}. Before that, we introduce the notion of periodic trajectory for PSVF. Actually, it is analogous to the definition of periodic trajectory for smooth systems.

\begin{definition}
Let $\Gamma_{Z}(t,q)$ a global trajectory of the PSVF $Z=(X,Y)$. We say that $\Gamma_{Z}$ is periodic if $\Gamma_{Z}$ is periodic in the variable $t$, i.e., if there exist $T>0$ such that $\Gamma_{Z}(t+T,q)=\Gamma_{Z}(t,p)$, for all $t\in\mathbb{R}$.
\end{definition}

\begin{theorem}\label{orbitas periodicas densas}
Consider $Z$ and $\Lambda$ as presented in Theorem \ref{teorema caos}. The periodic trajectories of $Z$ are dense in $\Lambda$.
\end{theorem}

\begin{proof}[Proof]
The proof is completed if we show that for every point $x\in\Lambda$ passes a periodic trajectory of $Z$. In order to see that, consider $\sigma_{0}$ the closed arc connecting point $x$ with itself ($\sigma_{0}\neq\{x\}$). The existence of such arc is due to Lemma \ref{lemma caos}. Then the global trajectory
$$
\Gamma_{Z}(t,x)=\displaystyle\bigcup_{i\in\mathbb{Z}}\{\sigma_{i}(t,x);t_{i}\leq t\leq t_{i+1}\}
$$
satisfying $\sigma_{i}=\sigma_{0}$ for all $i\in\mathbb{Z}$ is $t_{1}$-periodic and passes through $x$. Observe that $\sigma_{i}(kt_{1},x)=x$, for all $k\in\mathbb{Z}$ and for all $i\in\mathbb{Z}$.
\end{proof}



The concepts of chaos in PSVF are also discussed through the next section, where we present new examples of chaotic PSVF and an interesting relation between chaotic PSVF and orientable minimality.

\subsection{Orientable Minimality and Chaos}\label{classes_minimais}

In the recent theory of PSVFs there exist a lot of e\-xam\-ples where the dynamics in the non-smooth context is richer than their smooth analogous. It happens basically due to the non existence of a theorem that assure the uniqueness of a trajectory crossing the switching manifold. In fact, the orbit passing through a sliding or escaping segment on the switching manifold can run out from such segment when the time goes to future or past. Supported by these facts, in the present section we introduce some definitions concerning $Z$-minimal sets by distinguishing invariance for positive and negative global trajectories. That is what we call \textit{orientable minimality}. The advantage by taking into account such approach is to consider some interesting sets that are not properly $Z$-minimal but also present invariance and compactness in some sense.

\smallskip

In addition, we will verify the existence of some $Z$-minimal sets that do not have a \textit{canard structure}, i.e., coincidence of a visible and an invisible tangencies separating a sliding from an escaping region on the switching manifold. In particular, it means that the conditions in order to find a non-trivial $Z$-minimal set are not so strong as the situation when the canard phenomena happen. Finally, we must observe that some of these new objects also present a chaotic behaviour. This emphasizes that, in PSVFs, systems having non-trivial $Z$-minimal sets and chaotic behavior on them have some intersection, as we can see in Theorem \ref{chaos_minimality}.

\smallskip

In what follows we present new definitions on invariance and minimality. Then we compare some special sets taking into account such definitions.

\smallskip

\begin{definition}\label{invariancia+-}
A set $A\subset\mathbb{R}^{2}$ is $\mathbf{Z}$-\textbf{positively} (respectively, \textbf{negatively}) \textbf{invariant} if for each $p\in A$ and all positive global trajectory $\Gamma_{Z}^{+}(t,p)$ (respectively, negatively global trajectory $\Gamma_{Z}^{-}(t,p)$) passing through $p$ it holds $\Gamma_{Z}^{+}(t,p) \subset A$ (respectively, $\Gamma_{Z}^{-}(t,p) \subset A$).
\end{definition}

\smallskip

\begin{remark}\label{remark_a}
It follows directly from Definiton \eqref{invariancia+-} that a given set is invariant if and only if it is $Z$-positively and $Z$-negatively invariant.
\end{remark}

\begin{definition}\label{minimal+-}
Consider $\mathbf{Z \in \Omega}$. A set $M\subset\mathbb{R}^{2}$ is $Z$-\textbf{positively} (respectively, $Z$-\textbf{negatively}) \textbf{minimal} if
\begin{enumerate}
\item[(i)] $M\neq\emptyset$;
\item[(ii)] $M$ is compact;
\item[(iii)] $M$ is $Z$-positively (respectively, $Z$-negatively) invariant;
\item[(iv)] $M$ does not contain proper subset satisfying (i), (ii) and (iii).
\end{enumerate}
\end{definition}

The following lemma is a trivial consequence of Definition \ref{minimal+-}.

\begin{lemma}\label{lema_a}
Consider $M\in\mathbb{R}^{2}$ and $Z$ a PSVF. If $M$ is $Z$-positively minimal and $Z$-negatively minimal, then $M$ is $Z$-minimal.
\end{lemma}

\begin{proof}[Proof]
In fact, since $M$ is $Z$-positively minimal and $Z$-negatively minimal, then $M$ is a non-empty compact set and from Remark \ref{remark_a} $M$ is $Z$-invariant and does not contain a proper non-empty compact $Z$-invariant subset.
\end{proof}

Throughout this paper we will see that the converse of Lemma \ref{lema_a} does not holds.

\smallskip

We stress that a canard cycle of kind III is a closed $Z$-trajectory with, at least, a tangency point of $X$ or $Y$ with the switching
manifold (for a precise definition, see Definition 5 of \cite{BCE}). A pseudo-equilibria is any point $q$ such that $Z^{\Sigma}(q)=0$, where $Z^{\Sigma}$ represents the sliding vector field associated to $Z$.

\begin{example}\label{exemplo1}
We know that canard cycle of kind III and pseudo-equilibria are positive or negative invariants compact sets, but they are not positive and negative at the same time. It holds due to Definition \ref{definicao trajetorias} of local trajectories. Consequently they are $Z$-positively minimal or $Z$-negatively minimals, depending on the orientation of the trajectories, but not $Z$-minimal.

\begin{figure}[h!]
\begin{minipage}[b]{0.35\linewidth}
\begin{center}
\psfrag{A}{$\Gamma$}\psfrag{S}{$\Sigma$}
\epsfxsize=5cm  \epsfbox{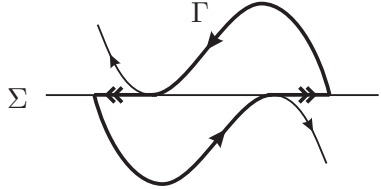}
\end{center}
\end{minipage}
\caption{\small{Canard cycle of kind III.}}\label{fig canard}
\end{figure}
\end{example}

In \cite{BCE} is exhibited examples of $Z$-minimal sets of PSVFs. In those examples it is always verified the occurrence of fold-fold singularities (which is characterized by the coincidence of a fold point of $X$ and a fold point of $Y$); however, at the next example of a $Z$-minimal set this fact is not required.

\begin{example}\label{exemplo2}
Consider $Z_{1}=(X,Y) \in \Omega$, where $X(x,y)=(1, - 2 x +1)$, $Y(x,y)=(-1,(-2 + x) (-22 + x (-7 + 4 x)))$ and $\Sigma=f^{-1}(0)=\{(x,y)\in\mathbb{R}^{2};y=0\}$. The parametric equation for the integral curves of $X$ and $Y$ with initial conditions $(x(0),y(0))=(0,k_{+})$ and $(x(0),y(0))=(0,k_{-})$, respectively, are known and their algebraic expressions are given by $y=-(-4 + x) (3 + x) + k_{+}$ and $y=(-4 + x) (-2 + x)^2 (3 + x) + k_{-}$, respectively. It is easy to see that $p=(1/2,0)$ is an invisible tangency point of $X$, $q=(2,0)$ is a visible tangency point of $Y$ and the points $p_{\pm}=((7 {\pm}\sqrt{401})/8,0)$ are both invisible tangency points of $Y$. Note that, in $\Sigma$, between $p_{-}$ and $p$ there exists an escaping region with a repeller pseudo equilibrium $\widetilde{p}$ on it and between $q$ and $p_{+}$ there exists a sliding region. Further, every point between $(-3,0)$ and $p_{-}$ or between $p_{+}$ and $q$ belong to a sewing region. Consider now the particular trajectories of $X$ and $Y$ for the cases when $k_{+}=0$ and $k_{-}=0$, respectively. These particular curves delimit a bounded region of plane that we call $\Lambda_{1}$. Figure \ref{fig minimal 1} summarizes these facts.
\end{example}

Example \ref{exemplo2} leads to the next proposition.

\begin{figure}[!h]
\begin{center}
\psfrag{A}{$p_-$}\psfrag{B}{$\widetilde{p}$}\psfrag{C}{$p$}\psfrag{D}{$q$}\psfrag{E}{$p_+$}\psfrag{S}{$\Sigma$}
 \epsfxsize=9cm
\epsfbox{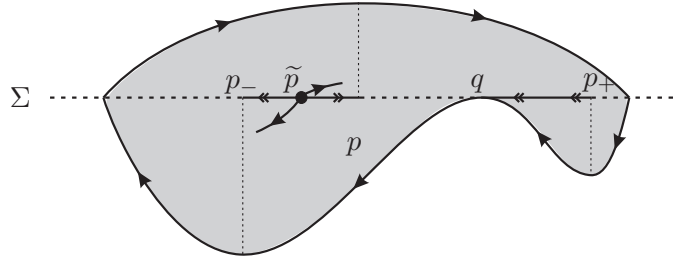}
\caption{\small{$Z_{1}$-minimal set $\Lambda_{1}$. $\Lambda_{1}$ is neither $Z_{1}$-positively minimal nor $Z_{1}$-negatively minimal.}} \label{fig minimal 1}
\end{center}
\end{figure}


\begin{proposition}\label{proposicao exemplo 2} Consider $Z_{1}=(X,Y) \in \Omega$, where $X(x,y)=(1, - 2 x +1)$, $Y(x,y)=(-1,(-2 + x) (-22 + x (-7 + 4 x)))$ and $\Sigma=f^{-1}(0)=\{(x,y)\in\mathbb{R}^{2};y=0\}$. The set
\begin{equation}\label{equacao conjunto minimal}
\begin{array}{ll}
\Lambda_{1}= & \{(x,y)\in\mathbb{R}^{2};-3\leq x\leq 4\;\; \mbox{and}\;\; \\
& (-4 + x) (-2 + x)^2 (3 + x)\leq y\leq -(-4 + x) (3 + x) \}.
\end{array}
\end{equation}
 is $Z_{1}$-minimal but it is neither $Z_{1}$-positively minimal nor $Z_{1}$-negatively minimal.
\end{proposition}

\begin{proof}[Proof]It is easy to see that $\Lambda_{1}$ is compact and has non-empty interior. Moreover, by Definition $\ref{definicao trajetorias}$, on $\partial \Lambda_{1} \setminus\{ q \}$ we have uniqueness of trajectory (here $\partial B$ means the boundary of the set $B$). Since the trajectory of $X$ passing through $p_-$ meets $\Sigma$ at $p_{-}^{1}=((1+\sqrt{401})/8,0)$, we conclude that the global trajectory of any point in $\Lambda_{1}$ meets $q$ for some time $t^{*}$. Since $q$ is a visible tangency point for $Y$, according to the fourth bullet of Definition \ref{definicao trajetorias} any trajectory passing through $q$ remain in $\Lambda_{1}$. Consequently $\Lambda_{1}$ is $Z$-invariant. Moreover, given $p_1, p_2 \in \Lambda_{1}$, with $p_2 \neq \widetilde{p}$ the positive global trajectory by $p_1$ reaches  $q$. The negative global trajectory by $p_2$ reaches $q$. So, there exists a global trajectory connecting $p_1$ and $p_2$. Now, let $\Lambda_{1}^{\prime} \subset \Lambda_{1}$ be a $Z$-invariant set. Given $q_1 \in \Lambda_{1}^{\prime}$ and $q_2 \in \Lambda_{1}$ since there exists a global trajectory connecting them we conclude that $q_2 \in \Lambda_{1}^{\prime}$. Therefore, $\Lambda_{1}^{\prime} = \Lambda_{1}$ and then $\Lambda_{1}$ is $Z_{1}$-minimal. Moreover, since $\partial \Lambda_{1}$ is $Z_{1}$-positively invariant, the set $\Lambda_{1}$ is not $Z_{1}$-positively minimal and since $\widetilde{p}$ is $Z_{1}$-negatively invariant the set $\Lambda_{1}$ is not $Z_{1}$-negatively minimal.
\end{proof}

\smallskip

\begin{example}\label{exemplo3}
The $Z$-minimal set presented in Theorem \ref{teorema caos} is also $Z$-positively minimal and $Z$-negatively minimal. The proof of this fact follows the same lines of the proof of Proposition \eqref{proposicao exemplo 2}.
\end{example}

\smallskip

The next example is a small variation of the Example \ref{exemplo3}.

\begin{example}\label{exemplo4}
Consider $Z_{2}$ a PSVF presenting the phase portrait exhibited in Figure \ref{fig minimal 3}. Here, there exists a compact set $\Lambda_2$ bounded by trajectories of $X$ and $Y$. As illustrated, $p_1$ and $p_3$ are invisible tangency points of $X$,  $p_2$ is a visible tangency point of $X$, $p_1$ and $p_3$ are visible tangency points of $Y$ and $p_0$, $p_2$ and $p_4$ are invisible tangency points of $Y$. It is easy to see that $\Lambda_2$ is $Z_{2}$-invariant and there is not a proper subset which is compact and $Z_{2}$-invariat. So, $\Lambda_2$ is $Z_{2}$-minimal. Assume that there exists a pseudo equilibrium $\widetilde{p}$ between $p_1$ and $p_2$. Following the orientation of the trajectories at Figure \ref{fig minimal 3} and the third bullet of Definition \ref{definicao trajetorias} we conclude that $\Lambda_2$ is not $Z_{2}$-negatively minimal since $\{ \widetilde{p} \}$ is a compact $Z_{2}$-negatively invariant set. Morevore, $\Lambda_2$  is $Z_{2}$-positively minimal since it is $Z_{2}$-positively invariant and it have not a compact proper subset which is $Z_{2}$-positively invariant.

\begin{figure}[!h]
\begin{center}
\psfrag{A}{$p_{0}$}\psfrag{B}{$p_{1}$}\psfrag{C}{$\widetilde{p}$}\psfrag{D}{$p_{2}$}\psfrag{E}{$p_{3}$}\psfrag{F}{$p_{4}$}\psfrag{S}{$\Sigma$}
 \epsfxsize=10cm
\epsfbox{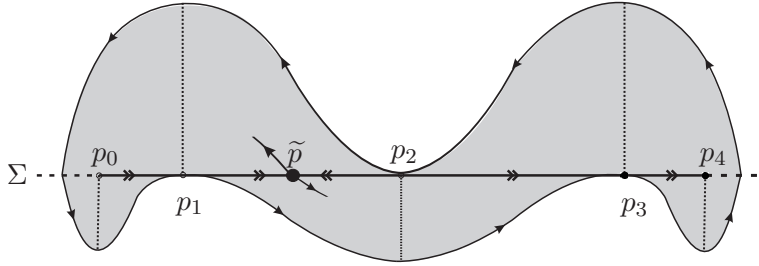} \caption{\small{The $Z_{2}$-minimal set $\Lambda_{2}$.}} \label{fig minimal 3}\end{center}
\end{figure}


\end{example}

\begin{proposition}\label{proposicao exemplo 4} Consider the notation of Example \ref{exemplo4}.  The set $\Lambda_2$
 is $Z_{2}$-minimal and also $Z_{2}$-positively minimal but not $Z_{2}$-negatively minimal for this PSVF.
\end{proposition}

\begin{proof}[Proof]
Straightforward, following the argumentation on Example \ref{exemplo4}.
\end{proof}

The next remark is an analogous of Proposition \ref{proposicao exemplo 4} by considering the opposite orientation of the time.

\begin{remark}\label{ultimo remark}
Consider a PSVF presenting the phase portrait exhibited in Figure \ref{fig minimal 3} with opposite orientation and the notation of Example \ref{exemplo4}. Then, following the same ideas of this example, we obtain that $\Lambda_2$
 is $Z_{2}$-minimal and $Z_{2}$-negatively minimal but not $Z_{2}$-negatively minimal for this PSVF.
\end{remark}

The previous propositions, examples and remarks of this subsection say that non-trivial minimality can occur in non symmetrical sets having no canard points. We must note that the example exhibited in \cite{J} of a chaotic planar system on a special set neither present  symmetry nor the set is minimal. Moreover, in our examples we can observe some similarity between symmetrical $Z$-minimal sets and systems that are $Z$-positively minimal and $Z$-negatively minimal simultaneously. Also, by observing Examples 1 to 4 we note that the presence of sliding and escaping regions on $\Sigma$ generates many different objects with very rich dynamics. We should note that the case where do not exist sliding or escape regions was studied in \cite{BCE} and was shown that in such case there are no non-trivial $Z$-minimal sets. In particular, we see that even considering different definitions of $Z$-minimal sets and non symmetrical sets having no canard points we can not generalize the classical Poincaré-Bendixson Theorem to the non-smooth context.

\smallskip

The following result indicates the presence of chaos in the systems studied in this section.

\begin{theorem}\label{teorema caos 2}
Consider the PSVF $Z_{2}$ and the set $\Lambda_{2}$ as presented in Proposition \ref{proposicao exemplo 4}. Then $Z_{2}$ is chaotic on $\Lambda_{2}$.
\end{theorem}

\begin{proof}[Proof]
The proof of Theorem \ref{teorema caos 2} follows the same lines of the proof of Theorem \ref{teorema caos} by using a similar result to Lemma \ref{lemma caos} for the $Z$-minimal set $\Lambda_{2}$.
\end{proof}

One should note that Theorems \ref{teorema caos} and \ref{teorema caos 2} present examples of PSVFs that are chaotic on minimal sets. This fact suggests a relation between chaoticity and minimality in PSVF that we make clear in the following theorem.

\smallskip

In what follows we denote by $med(\cdot)$ the Lebesgue measure.

\begin{theorem}\label{chaos_minimality}
Let $Z$ be a planar PSVF and $\Lambda\subset\mathbb{R}^{2}$ a compact invariant set. If $\Lambda$ is $Z$-positively minimal and $Z$-negatively minimal satisfying $med(\Lambda)>0$, then $Z$ is chaotic on $\Lambda$.
\end{theorem}

Theorem \ref{chaos_minimality} is a very interesting result because presents a connection between two important different objects of the recent theory of PSVF, namely, the chaotic planar systems and the non-trivial $Z$-minimal sets.

\smallskip

In order to prove Theorem \ref{chaos_minimality}, we introduce the next two lemmas. The first one is a generalization of Lemma \ref{lemma caos}.

\begin{lemma}\label{lema conexao pontos}
Under the same hypotheses of Theorem \ref{chaos_minimality}, it holds that for any $x,y\in\Lambda$, there exist a global trajectory $\Gamma(t,y)$ passing through $y$ and $t^{*}>0$ such that $\Gamma^{+}(t^{*},y)=x$.
\end{lemma}

\begin{proof}[Proof]
Since $med(\Lambda)>0$, by Poincaré-Bendixson Theorem for PSVF presented in \cite{BCE}, there exist at least a set $A\subset\Sigma\cap(\Sigma^{e}\cup\Sigma^{s})$. Othewise, we have $\Sigma\cap\Lambda=\Sigma^{c}\cup\Sigma^{t}$ and then by the referred theorem we get $med(\Lambda)=0$, where $\Sigma^{t}$ is the set of tangencies points of $Z$. For each $a\in A$, denote by $\Pi^{+}_{a}$ the set of all positive global trajectories passing through $a$ and by $\Pi^{-}_{a}$ its negative analogous. Now consider the sets
$$
A^{\pm}_{a}=\displaystyle\bigcup_{\Gamma_{a}\in\Pi^{\pm}_{a}}\Gamma_{a}(t,a)\subset\Lambda.
$$
Actually we have $A^{\pm}_{a}=\Lambda$, since $A^{\pm}_{a}$ is $Z$-positively (respectively negatively) invariant restrained in the $Z$-positively (respectively negatively) minimal set $\Lambda$. In order to see that $A^{+}_{a}$ is $Z$-positively invariant, let $p$ be a point in $A^{+}_{a}$ and $\Gamma_{p}(t,p)$ a positive global trajectory passing through $p$. Since $p\in A^{+}_{a}$, then there exists a positive global trajectory $\tilde{\Gamma}_{a}(t,a)$ passing through $a$ and $t_{0}>0$ such that $\tilde{\Gamma}_{a}(t_{0},a)=p$. Consequently $\Gamma_{p}(t,p)$ belongs to $A^{+}_{a}$ once it is restrained to the positive global trajetory $\hat{\Gamma}_{a}(t,a)=\tilde{\Gamma}_{a}(t,a)\cup\Gamma_{p}(t,p)\subset A^{+}_{a}$. Analogously we can prove that $A^{-}_{a}$ is $Z$-negatively invariant.

\smallskip

Now consider $x,y\in\Lambda$ arbitrary points. Since $A^{-}_{a}=\Lambda=A^{+}_{a}$, there exists $\Gamma_{a}^{+}(t,a)\in A^{+}_{a}$ a positive global trajectory, $\Gamma_{a}^{-}(t,a)\in A^{-}_{a}$ a negative global trajectory and values $t_{x}>0$, $t_{y}<0$ such that $\Gamma_{a}^{+}(t_{x},a)=x$ and $\Gamma_{a}^{+}(t_{y},a)=y$. Consequently there exists a global trajectory $\Gamma(t,y)$ passing through $y$ and $t^{*}=t_{x}+|t_{y}|>0$ such that $\Gamma(t^{*},y)=x$.
\end{proof}

\begin{lemma}\label{lema conexao caos}
Under the same hypotheses of Theorem \ref{chaos_minimality}, if any two points of $\Lambda$ can be connected by a global trajectory of $Z$, then $Z$ is chaotic on $\Lambda$.
\end{lemma}

\begin{proof}[Proof]
The proof of Lemma \ref{lema conexao caos} is similar to the proof of Theorem \ref{teorema caos} by using Lemma \ref{lema conexao pontos} instead of Lemma \ref{lemma caos}.
\end{proof}

\begin{proof}[Proof of Theorem \ref{chaos_minimality}]
The proof is straightforward from Lemmas \ref{lema conexao pontos} and \ref{lema conexao caos}.
\end{proof}

\begin{remark}
One should note that we can not change the hypotheses of Theorem \ref{chaos_minimality} by considering $Z$-minimal sets instead of $Z$-positively minimal sets and $Z$-negatively minimal sets simultaneously. Indeed, consider the PSVF $Z_{1}$ and the set $\Lambda_{1}$ as presented in Proposition \ref{proposicao exemplo 2}. It holds that $\Lambda_{1}$ is $Z_{1}$-minimal. Nevetheless, $Z_{1}$ is not chaotic on $\Lambda_{1}$, since it is not topologically transitive on $\Lambda_{1}$. In order to see that, consider a nonempty open set $U$ located in $\Sigma^{+}$ just above the sliding segment $S$ between $q$ and $p_{+}$ in such way that all points of $U$ reach $S$ from $\Sigma^{+}$ to $\Sigma$ and do not enter in the region $\Sigma^{-}\setminus\Sigma$. Consider also a nonempty open set $V$ under the same conditions of $U$, however located under $S$ on $\Sigma^{-}$. Thus it is clear that all points of $U$ and $V$ reach $S$ and slides to $\partial \Lambda_{1}$ through the point $q$. However, since $\partial \Lambda_{1}$ is $Z_{1}$-positively minimal, it follows that the trajectories of $U$ and $V$ do not escape from $\partial \Lambda_{1}$ for positive values of time. Consequently we can not connect points of $U$ and $V$ through a positive global trajectory and therefore $Z_{1}$ is not topologically transitive on $\Lambda_{1}$.

\end{remark}

\subsection{Conclusion}\label{conclusao}

In this paper we have verified the existence of non deterministic chaos in planar PSVFs whithout symmetry or presenting a canard structure. As far as the authors know, this is the first time that non-smooth systems with such characteristic are observed in the planar case. Moreover, we introduce definitions of $Z$-minimal sets of PSVFs taking into account the fact that PSVFs have a strong dependence of the orientation of the trajectories, as we can see in Definition \ref{definicao trajetorias}. Finally, we verify the presence of chaotic behavior in planar PSVFs and present a result re\-la\-ting chaotic behaviour with orientable minimality, which emphasizes the importance of providing the definition of orientable minimality.\\

\noindent {\textbf{Acknowledgments.} The first and second authors are partially supported by the CNPq-BRAZIL grant 478230/2013-3. The second author is partially supported by the FAPESP-BRAZIL grant 2012/00481-6. The third author is supported by the FAPESP-BRAZIL grant 2010/18015-6}.

\end{document}